\documentclass[11pt,fleqn]{article}
\usepackage{latexsym}
\usepackage{amsmath,amsthm,amssymb}
\usepackage{epsfig,overpic}
\usepackage{amsfonts}

\begin{document}

%\parskip\baselineskip
%\makeatletter
%\if\c@secnumdepth>0
\newtheorem{definition}{Definition}[section]
%\else
%\newtheorem{definition}{Definition}
%\fi
%\makeatother
%\newtheorem{definition}{Definition}[subsection]
\newtheorem{theorem}[definition]{Theorem}
\newtheorem{lemma}[definition]{Lemma}
\newtheorem{proposition}[definition]{Proposition}
\newtheorem{corollary}[definition]{Corollary}

\theoremstyle{definition}
\newtheorem{definitions}[definition]{Definitions}
\newtheorem{remark}[definition]{Remark}
\newtheorem{remarks}[definition]{Remarks}
\newtheorem{example}[definition]{Example}
\newtheorem{examples}[definition]{Examples}
\newtheorem{leeg}[definition]{}
\newtheorem{comments}[definition]{Some comments}
\def\square{\Box}
%\newremark{example}[definition]{Example}
%\newremark{examples}[definition]{Examples}
%\newtheorem{remark}[definition]{Remark}
\newtheorem{observation}[definition]{Observation}
\newtheorem{defsobs}[definition]{Definitions and Observations}
%\newremark{assumption}[definition]{Assumption}
\newenvironment{prf}[1]{ \trivlist
\item[\hskip \labelsep{\it
#1.\hspace*{.3em}}]}{~\hspace{\fill}~$\square$\endtrivlist}

\title{Order one equations with the Painlev\'e property}
\author{Georg Muntingh $\ $ and $\ $ Marius van der Put\\
\footnotesize Department of Mathematics, University of Groningen,
 P.O.Box 800,\\
\footnotesize 9700 AV Groningen, The Netherlands, georg.muntingh@gmail.com and mvdput@math.rug.nl}
\date{}
\maketitle
\noindent

\begin{abstract}
% Motivation
\noindent Differential equations with the Painlev\'e property have been studied extensively due to their appearance in many branches of mathematics and their applicability in physics.
% Problem statement
Although a modern, differential algebraic treatment of the order one equations appeared before, the connection with the classical theory did not.
% Approach, results (and conclusions?)
Using techniques from algebraic geometry we provide the link between the classical and the modern treatment, and with the help of differential Galois theory a new classification is derived, both for characteristic 0 and $p$.
\end{abstract}

\section{Introduction}
The solutions $y=y(z)$ of a complex differential equation of order one of the form $f(y',y,z)=0$, where $y':=\frac{dy}{dz}(z)$ and $f$ is a polynomial in the first two variables and a rational or algebraic or meromorphic function in the last variable, can have poles, branch points or essential singularities. The equation $f$ is said to have {\em movable singularities} if the set of branch points or the set of essential singularities is not discrete (note that the literature contains various other definitions). In the opposite case $f$ is said to have the Painlev\'e property (PP). The purpose of this paper is to give a precise classification of the order one equations with the PP, including complete proofs.

The classical literature starts halfway the nineteenth century with a series of papers by Briot and Bouquet \cite{BriotBouquet2,BriotBouquet1} in which they treat the autonomous case, i.e., $f(y',y) = 0$. The general case was solved by L. Fuchs \cite{Fuchs}, H. Poincar\'e \cite{Poincare} and P. Painlev\'e \cite{Painleve1888}, and essentially correct statements can be found in Painlev\'e's \emph{Stockholm Lectures} \cite{Painleve} and in an article by J. Malmquist \cite{Malmquist}. The classical proofs of Malmquist are, however, incomprehensible for the modern reader. Furthermore, Ince \cite{Ince}, Bieberbach \cite{Bieberbach} and Hille \cite{Hille} discuss first order equations with the PP. Unfortunately, their statements and proofs are far from complete and precise.

In 1980 M. Matsuda published a book \cite{Matsuda} on first order algebraic differential equations that deserves to be better known. For any characteristic, Matsuda gives a purely algebraic definition for a differential field to have `no movable singularities,' and then gives an essentially correct description of differential fields with this property. However, the link with the above definition of `no movable singularities' is not hinted at.

Now we describe the method of this paper. In the autonomous case, the equation
$f(y,y')=0$ defines a curve $X$ over $\mathbb{C}$ and a $\mathbb{C}$-linear
derivation $D$ on its field of functions $\mathbb{C}(X)$. In other words,
$D$ is a (meromorphic) vector field on $X$. The condition that $f$ has only
finitely many branch points turns out to be equivalent to `$D$ has no poles'. 
The pairs $(X,D)$ with this property are easily classified and have the PP.

Consider a differential field $(K,\frac{d}{dz})$ which is a finite extension
of $(\mathbb{C}(z),\frac{d}{dz})$ and thus $K$ is the function field of a 
curve $Z$ over $\mathbb{C}$. (In fact, the proofs are valid for more general
differential fields).
 An equation $f(y,y')=0$ over $K$ defines a
curve $X$ over $K$ and a derivation $D$ on its field of functions $K(X)$
extending $\frac{d}{dz}$ on $K$. In general, $D$ has finitely many poles on
$X$, i.e., closed points $P$ of $X$ such that $D(O_{X,P})\not \subset
O_{X,P}$. In particular $D$ can be interpreted as a meromorphic connection on
the line bundle $O_X$. Prop. 4.2 states that there are infinitely many branch
points for $f$ on $Z$ if $D$ has a pole (this is missing in \cite{Matsuda}). 

\smallskip

We illustrate this by the example $(y')^2=y-z^2$ over the differential field 
$K=(\mathbb{C}(z),\ \frac{d}{dz})$. Write $t=y,\ s=y'$ and consider the
differential field $F=\mathbb{C}(z)(s,t)$ with equation $s^2=t-z^2$ and
with the $\mathbb{C}$-linear derivation $D$ given by $D(z)=1,\ D(t)=s$.
Thus $F=\mathbb{C}(z)(s)$ and $D(s)=\frac{1}{2}-\frac{z}{s}$. We view $F$
as the function field of the curve $X=\mathbb{P}^1_K$, i.e., the projective
line over $K$. The derivation $D$ has a pole at the point $s=0$, which means
that the local ring $O_{X,0}=K[s]_{(s)}$ at the point $s=0$ is mapped outside
itself (indeed, $D(s)\not \in O_{X,0}$). This is the only pole of $D$ because
for $u=1/s$ one has $D(u)=-u^2/2+zu^3$. Prop. 4.2 claims the existence of
infinitely many values $z_0$ such that there exists a branched
solution passing through $z_0$, i.e., a solution $y$ using roots of $(z-z_0)$.

 We verify this explicitly be looking at the solutions of the equation
as graphs in the $(y,z)$ plane. For any point $(y_0,z_0)$ such that 
$y_0-z_0^2\neq 0$ there are locally two holomorphic solutions corresponding to 
the choices for $y'(z_0)$ satisfying $y'(z_0)^2=y_0-z_0^2$. Also for the point 
$(0,0)$ one finds two holomorphic solutions, namely $y=az^2$ with 
$4a^2-a+1=0$. At any point $(y_0,z_0)\neq (0,0)$ with $y_0-z_0^2=0$ there are
locally two branched analytic solutions of the form
$y=y_0+c(z-z_0)^{3/2}+\cdots$ with $c^2=-\frac{8}{9}z_0$. We conclude that at 
any point $z_0\neq 0$ there passes a branched solution.

\bigskip

For the classification of the pairs $(X,D)$ where $D$ has no poles, the field
$K$ can be any differential field and may even have positive characteristic   
(this is the theme of \cite{Matsuda}). In the Appendix we treat the case of 
positive characteristic. Our main result, Thm. 4.5, states that, 
after a finite extension of $K$, the curve $X$ is defined over 
$\mathbb{C}$. However, the derivation $D$ may use elements of $K$. In all
cases the equation has the PP. 

For curves $X$ of genus 0 or 1 an explicit elementary proof is given. For
higher genus, differential Galois theory is applied to the connection
$\nabla$, derived from $D$ on $O_X$, on (powers of) the sheaf of relative
differential forms $\Omega _{X/K}$ on $X$, to prove that, after a
finite  extension of $K$, $X$ is defined over $\mathbb{C}$. 
Finally we note that in most cases these finite (separable) extensions of the
differential field $K$ are necessary for the main result.

\section{Autonomous Case}
Given is a polynomial $f(S,T)\in \mathbb{C}[S,T]$ involving both $S$ and $T$. For the question of whether the equation $f(y',y)=0$ has movable singularities we may and will suppose that $f$ is irreducible. We use the following notation: $\mathbb{C}\{z\}$ is the ring of convergent power series in $z$ (i.e., the germs at $z=0$ of holomorphic functions) and $\mathbb{C}(\{z\})$ denotes its field of fractions (the convergent Laurent series at $z=0$). For any integer $m\geq 1$ one considers the ring $\mathbb{C}\{z^{1/m}\}$ and its field of fraction $\mathbb{C}(\{z^{1/m}\})$.

By a {\it branched solution} for $f$ (at $z=0$) we will mean a $y\in \mathbb{C}(\{z^{1/m}\})$ with $m>1$ and $y\not \in \mathbb{C}(\{z\})$, satisfying $f(y',y)=0$. Since the equation is autonomous, the existence of such a $y$ will make any point of the complex plane into a branch point for $f$ and $f$ does not have the PP. On the other hand we will classify the equations $f$ having no branched solution (at $z=0$) and show that they have the PP.

One considers the field $F:=\mathbb{C}(s,t)$, defined by the equation $f(S,T)=0$. It is the function field of a certain curve $X$ over $\mathbb{C}$ (or compact Riemann surface). One considers on $F$ the $\mathbb{C}$-linear derivation $D$ given by $D(t)=s$. This makes $F$ into a differential field. Note that $D$ has an obvious interpretation as (meromorphic) vector field on $X$. The fields $\mathbb{C}(\{z^{1/m}\})$ are considered as differential fields w.r.t. the derivation $\frac{d}{dz}$. A branched solution is equivalent to a $\mathbb{C}$-linear homomorphism of differential fields $\phi :(F,D)\rightarrow (\mathbb{C}(\{z^{1/m}\}),\frac{d}{dz})$ with $m>1$ and such that the image of $\phi$ is not contained in $\mathbb{C}(\{z\})$.

\begin{theorem}\label{thm:Autonomous}~\\
  \begin{itemize}
  \item[(1)] $f$ has the PP if and only if the derivation $D$ on the field $F=\mathbb{C}(s,t)$ has no poles (i.e., the vector field $D$ has no poles).
  \item[(2)] The only possibilities for $(\mathbb{C}(s,t),D)$ with the PP are:
    \begin{itemize}
    \item[(a)] $(\mathbb{C}(x),(a_0+a_1x+a_2x^2)\frac{d}{dx})$ (with $a_0,a_1, a_2\in \mathbb{C}$ not all $0$) and \\
    \item[(b)] $(\mathbb{C}(x,y),D)$ with $y^2=x^3+ax+b$ (non singular elliptic curve) and $D=cy\frac{d}{dx}$ with $c\in \mathbb{C}^*$.
    \end{itemize}
  \end{itemize}
\end{theorem}
\begin{proof} (1) Suppose that there exists a branched solution, represented  by a homomorphism of differential fields $\phi :F\rightarrow \mathbb{C}(\{z^{1/m}\})$ with $m>1$ and $m$ minimal. Then $\phi$ induces a discrete valuation on $F$ and thus a (closed) point $x\in X$.

Let $O_{X,x}$ be the local ring of the point $x$. By definition $\phi$ maps $O_{X,x}$ to $\mathbb{C}\{z^{1/m}\}$. Let $p$ be a local parameter for $O_{X,x}$. Then $O_{X,x}^{an}$, the analytic local ring at the point $x$ of $X$ (seen as Riemann surface), is equal to $\mathbb{C}\{p\}$. Now $\phi$ uniquely extends to a homomorphism $\phi^{an} : O_{X,x}^{an} \rightarrow \mathbb{C}\{z^{1/m}\}$. Indeed, $\phi^{an}$ is given by the formula $\phi^{an} (\sum_{n\geq 0} a_n p^n) = \sum_{n\geq 0} a_n \phi (p)^n$.

The minimality of $m$ implies that $\phi^{an}$ is an isomorphism. Hence $\phi^{an}$ induces an isomorphism $\mathbb{C}(\{p\})\rightarrow \mathbb{C}(\{z^{1/m}\})$. The derivation $D$ on $F$ induces a derivation $D^{an}:\mathbb{C}\{p\}\rightarrow \mathbb{C}(\{p\})$
by the formula $D^{an}(\sum _{n\geq 0}a_np^n)=\sum _{n\geq 0}na_np^{n-1}\cdot D(p)$.

From the formula $\phi \circ D=\frac{d}{dz}\circ \phi$ and the definitions of $\phi ^{an}$ and $D^{an}$ one deduces $\phi ^{an}\circ D^{an}=\frac{d}{dz}\circ \phi ^{an}$. Finally, from $\frac{d}{dz}z^{1/m}=1/m\cdot z^{-1+1/m}\not \in \mathbb{C}\{z^{1/m}\}$, one deduces that $D(p)\not \in O_{X,x}$. Thus $D$ has a pole at $x$.

On the other hand, suppose that the derivation $D$ has a pole at the point $x\in X$. Then $D(O_{X,x})$ is not contained in $O_{X,x}$. As above, we extend $D$ to a derivation $D^{an} : O_{X,x}^{an} = \mathbb{C}\{p\} \rightarrow \mathbb{C}(\{p\})$. Then $D=a(p)\frac{d}{dp}$ for some $a(p)\in \mathbb{C}(\{p\})$ which has a pole of order $1-m$ (with $m>1$).

We want to solve the equation $a(p)\frac{d}{dp}v = \frac{1}{m}v^{1-m}$ where $v$ is a local parameter for $\mathbb{C}\{ p \}$ (i.e., $v=c_1p+c_2p^2+\cdots $ with $c_1\neq 0$). One rewrites the equation as $\frac{d}{dp}(v^m)=\frac{1}{a(p)}$. There is a $b(p)\in \mathbb{C}\{p\}$ of order $m$ with $\frac{d}{dp}b(p) = \frac{1}{a(p)}$. Then $z:=b(p)$ has an $m$th root $v\in C\{ p\}$. The order of $v$ is 1 and thus $v$ is a local parameter. One concludes that $(\mathbb{C}(\{p\}), D^{an}) = (\mathbb{C}(\{z^{1/m}\}),\frac{d}{dz})$. This defines a branched solution.

(2) Using Riemann-Roch one finds that the sheaf of derivations on $X$ has a non zero global section only if the genus of $X$ is 0 or 1. In the first case the field $F=\mathbb{C}(x)$ and $D=(a_0+a_1x+a_2x^2)\frac{d}{dx}$ is the space of the derivations without poles. For genus 1, the field is as in the statement and the only (non zero) derivations without poles are $cy\frac{d}{dx}$ with $c\in \mathbb{C}^*$.
\end{proof}

\begin{example}
Let $f := S^p-T^q$ with $(p,q) = 1$. The field $\mathbb{C}(s,t)$ equals $\mathbb{C}(x)$ where $s=x^q$, $t=x^p$. The formula $D(t)=s$ and the last part of the theorem imply that PP is equivalent to $q-p\in \{-1,0,1\}$.
\end{example}

\begin{corollary}
The Riccati differential equation and the Weierstrass differential equation, i.e., the cases (a) and (b) of Theorem \ref{thm:Autonomous},  have the PP.
\end{corollary}
\begin{proof}
We quickly give the well known arguments. A solution to (a) or (b) corresponds to a differential homomorphism  $\phi$ from $(F,D)$ to some differential ring of functions. In case (a) this means that $g=\phi (x)$ satisfies the Riccati equation $g'=a_0+a_1g+a_2g^2$. In case (b) this means that $g=\phi (y)$ satisfies a Weierstrass equation $(\lambda g')^2=g^3+ag+b$. Case (a) with $a_2=0$ is an inhomogeneous linear equation and the solutions have no other singularities than infinity. Case (a) with $a_2\neq 0$ is the Riccati equation associated to a linear equation of order two. The solutions of the Riccati equation have the form $g=\frac{v'}{v}$, where $v$ is a non zero solution of this linear equation. Thus $g$ can at most have a pole. Finally, the solutions to case (b) are shifts of the Weierstrass function and the singularities in the finite plane are poles.
\end{proof}

\noindent {\it Remark}. In the paper of Briot and Bouquet, the PP is translated
into an explicit, necessary and sufficient condition for the irreducible
polynomial $f\in \mathbb{C}[S,T]$ \cite{BriotBouquet2}. It can be verified that this condition is equivalent to the cases (a) or (b) of Theorem \ref{thm:Autonomous}.

\section{Assumptions and Notation}
In the general case one considers a differential equation of the form $f(y',y) = 0$ where $f(S,T)$ is a polynomial in $S$ and $T$ and has coefficients in some field $K$ (possibly of germs) of meromorphic functions, defined on some part of the complex plane and closed under $\frac{d}{dz}$. We restrict ourselves to the most interesting case where $K$ is a finite extension of $\mathbb{C}(z)$. One can copy the proof below to the case of other fields of meromorphic functions.

Thus $K$ is the field of meromorphic functions on some compact Riemann surface $Z$. A closed point $P$ of $Z$ (or place of $K$) with local parameter $p$ is called {\it ramified for $f$} if there exists a $y\in \mathbb{C}(\{p^{1/e}\})$, not an element of $\mathbb{C}(\{p\})$, satisfying $f(y',y)=0$. Further we will call $y$ a {\it branched solution at $P$}. If there are infinitely many ramified points $P$ for $f$, then certainly $f$ does not have the PP. In what follows we will classify the equations $f$ that have only finitely many ramified points on $Z$ and show that these equations do have the PP. We note that replacing $K$ by a finite extension does not change the property: ``there are only finitely many ramified points for $f$.'' In particular, we will assume that $f\in K[S,T]$ is absolutely irreducible, {\it i.e.,} $f$ is irreducible as element of $\overline{K}[S,T]$. Further we will assume that both $S$ and $T$ are present in $f(S,T)$.

Write $K[s,t]=K[S,T]/(f)$. To a branched solution $y$ at $P$ (with local parameter $p$) one associates a homomorphism  $\psi :K[s,t]\rightarrow \mathbb{C}(\{p^{1/e}\})$, extending the given embedding $K\subset \mathbb{C}(\{p\})$, by sending $t$ to $y$ and $s$ to $y':=\frac{d}{dz}y$. One would like $K[s,t]$ to be a differential ring by introducing the $\mathbb{C}$-linear derivation $D$ by $D(z)=1$ and $D(t)=s$. Let $\Delta$ denote $\frac{df}{ds}$ (evaluated at $s,t$). In general $K[s,t]$ is not invariant under $D$, but the localization $K[s,t,\frac{1}{\Delta}]$ is seen to be invariant under $D$ and is therefore a differential ring.

There are only finitely many $y$ (obviously algebraic over $K$) such that the corresponding $\psi$ satisfies $\psi (\Delta)=0$. Omitting these $y$, we may extend the given $\psi :K[s,t]\rightarrow \mathbb{C}(\{p^{1/e}\})$ to a $\psi$ defined on the localization $K[s,t,\frac{1}{\Delta}]$. This new $\psi$ has a non trivial kernel, if and only if $y$ is algebraic over $K$. In this case $K(y)$ is a finite extension of $K$ and the kernel of $\psi$ is a maximal ideal of $K[s,t,\frac{1}{\Delta}]$, closed under the differentiation $D$.

Consider a branched solution $y$ such that the corresponding $\psi$ has a trivial kernel (thus $y$ is transcendental over $K$). Let $F$ denote the field of fractions of $K[s,t]$, made into a differential field by the unique extension (also called $D$) of the derivation $D$. Then $\psi$ extends to a homomorphism of differential fields $\phi :(F,D)\rightarrow (\mathbb{C}(\{p^{1/e}\}),\frac{d}{dz})$, where $e > 1$ is the smallest natural number for which the image of $\phi$ is contained in $\mathbb{C}(\{p^{1/e}\})$. This is called a {\it transcendental branched solution}. One associates to $F$ the non singular, absolutely irreducible, projective curve $X$ over $K$ with function field $F$.

\section{General Case}
We will need the following result.

\begin{lemma}\label{lem:BranchedSolution}
Let $m>1$ and $f\in {\bf C}\{z,w\}$ with $f(0,0)\neq 0$ be given. Choose $c\in \mathbb{C}^*$ with $c^m=f(0,0)$. Then the equation $(y^m)'=f(z,y)$ has a solution in $\mathbb{C}\{z^{1/m}\}$ of the form $y=cz^{1/m}+\cdots$.
\end{lemma}

\begin{proof}
Consider a formal series $y=\sum _{i\geq 1}c_iz^{i/m}$. Then 
\[(y^m)'=\sum _{a\geq m}\frac{a}{m} \{\sum _{i_1+\cdots +i_m=a}c_{i_1}\cdots c_{i_m}\} \cdot z^{-1+a/m}\ .\]
Comparing this expression with $f(z,\sum c_iz^{i/m})$ one finds $c_1^m=f(0,0)$. After choosing $c_1=c$, one compares the coefficients of $z^{-1+a/m}$ and one finds that $c^{m-1}c_{a-m+1}$ equals a polynomial formula in $c_1,\dots ,c_{a-m}$.  Thus we found a unique formal series satisfying the differential equation. One can verify (by brute force) that this series is in fact convergent. More explicitly, we may suppose $f(0,0)=1$ and $c=1$.  Write (for the moment) $z=t^m$ and $y=t(1+h)$ with $h\in t\mathbb{C}[[t]]$. One obtains a differential equation for $h$ of the form
\[t\frac{d}{dt}h+mh=\sum _{a\geq 1,b\geq 0}c_{a,b}t^ah^b + \sum _{b\geq 2}c_bh^b\ ,\]
where the coefficients satisfy $|c_{a,b}|\leq R^{a+b}$ and $|c_b|\leq R^b$ for some $R>0$. After multiplying $h$ and $t$ by suitable constants we may suppose that $R$ is sufficiently small. Write $h=\sum _{n\geq 1}h_nt^n$, then the recurrence relation for the $h_n$ takes the form $(n+m)h_n=$ a polynomial in $h_1,\dots,h_{n-1}$ of which the absolute value can be estimated. One can verify that
$\lim h_n=0$.
\end{proof}

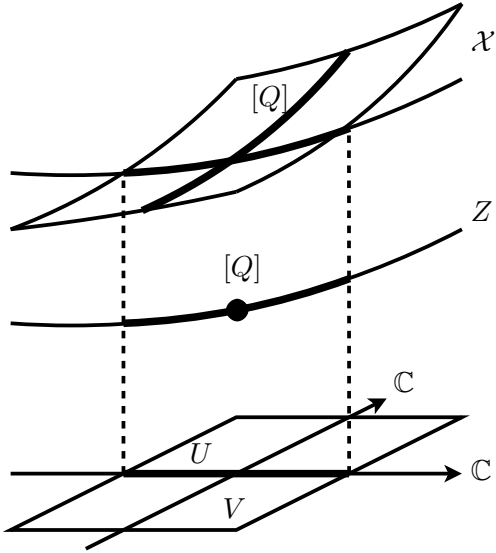
\begin{figure}
\begin{center}
\begin{overpic}[scale=.5]{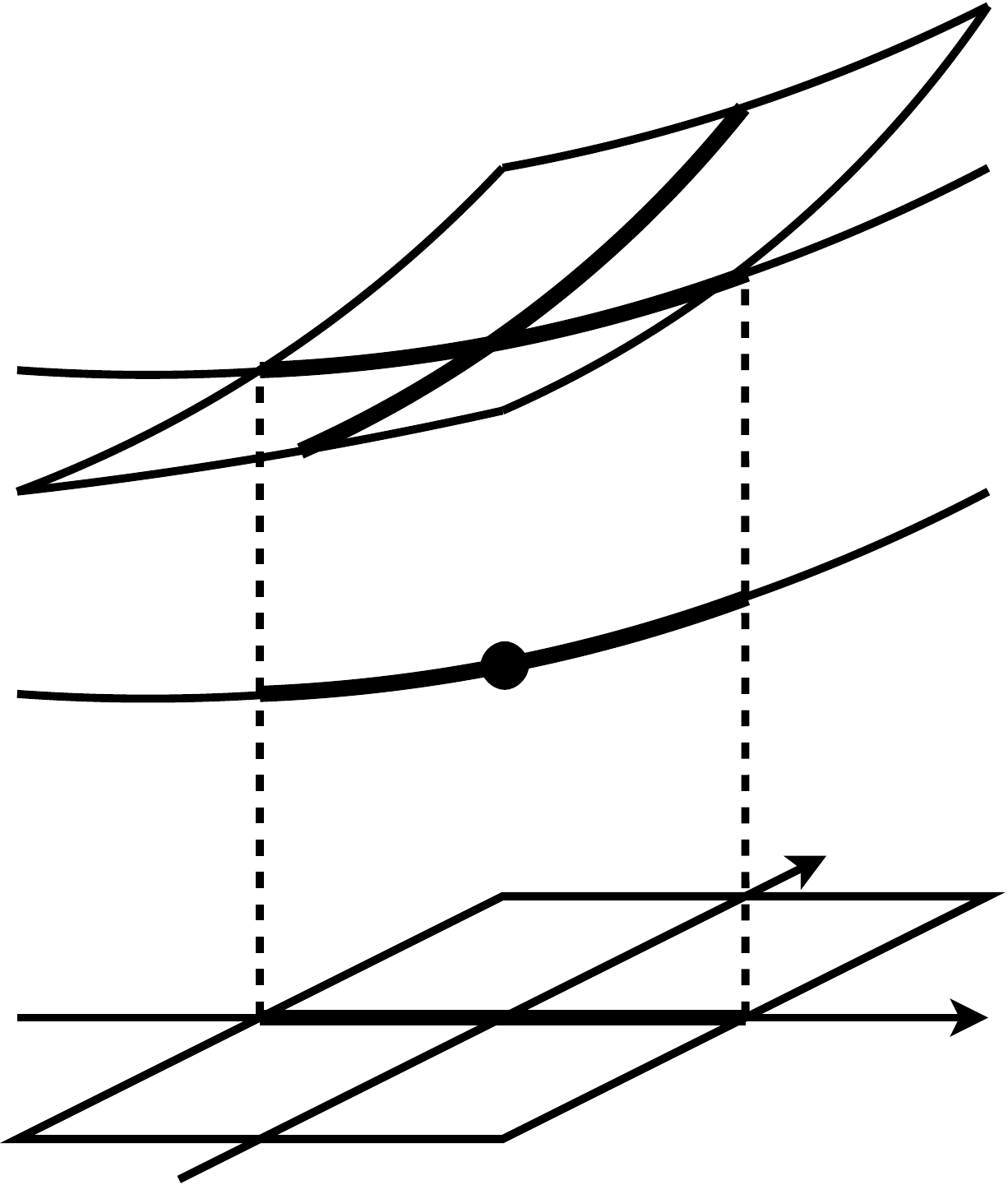}
\put(34,16){$U$}
\put(40,6){$V$}
\put(40,50){$[Q]$}
\put(45,81){$[Q]$}
\put(71,29){$\mathbb{C}$}
\put(85,13){$\mathbb{C}$}
\put(85,60){$Z$}
\put(85,91){$\mathcal{X}$}
\end{overpic}
\end{center}
\caption{A schematic representation of the situation in the proof of Proposition \ref{prop:FiniteBranchPointsImpliesRegular}.}
\end{figure}

\begin{proposition}\label{prop:FiniteBranchPointsImpliesRegular}
Suppose that there exists a closed point $Q$ of $X$ such that the local ring $O_{X,Q}$ is not invariant under $D$. Then there are infinitely many ramified points {\rm (}on $Z${\rm )} for $f$.
\end{proposition}
\begin{proof}
The assumption and the statement that we want to prove are stable under finite extensions of $K$. Thus we may suppose that the closed point $Q$ is rational over $K$. The field $F$ can also be seen as the function field of some normal projective surface $\mathcal{X}$ over $\mathbb{C}$. The inclusion $K\subset F$ induces a `rational map' $\mathcal{X}\dashrightarrow Z$, where $Z$ is the non singular, irreducible, projective curve over $\mathbb{C}$ with function field $K$. The assumption $Q\in X(K)$ induces a `rational section' of the above rational map. In the analytic category, one can avoid singularities and describe locally this rational map and section by the following:

$U\subset Z$ is a disk, identified with $\{z\in \mathbb{C}|\ |z-z_0|<\epsilon \}$. $V\subset \mathcal{X}$ is a multidisk $\{z\in \mathbb{C}|\ |z-z_0|<\epsilon \} \times \{u\in \mathbb{C}|\ |u|<\epsilon \}$. The rational map is just $(z,u)\mapsto z$ and the rational section is $z\mapsto (z,0)$. Further, $u = 0$ is the restriction to $V$ of the divisor on $\mathcal{X}$ given by the point $Q$ on $X$. We note that $u$ is a local parameter at $Q$ and the completion of $O_{X,Q}$ equals $K[[u]]$. It is given that $D(u)\not \in K[[u]]$. Therefore $D(u)=a_{-m+1}u^{-m +1}+\cdots$ with $m>1$ and $a_{-m+1}\in K$, different from 0. Then $D(u^m)\in K[[u]]$ and does not lie in $uK[[u]]$.

In analytic terms, $D(u^m)$ is a non zero holomorphic function on $V$. After shrinking $U$ and $V$ we may suppose that $D(u^m)$ has no zeros on $V$. Take now any point $z_1\in U$. Then we have $D(u^m)=g(z-z_1,u)\in \mathbb{C}\{z-z_1,u\}$ with $g(0,0)\neq 0$.

According to Lemma \ref{lem:BranchedSolution}, the differential equation $(h^m)'=g(z-z_1,h)$ has a solution $h=*(z-z_1)^{1/m}+\cdots$ lying in $\mathbb{C}(\{(z-z_1)^{1/m}\})$. The homomorphism $\mathbb{C}\{z-z_1,u\}\rightarrow \mathbb{C}\{(z-z_1)^{1/m}\}$ given by $z-z_1\mapsto z-z_1,\ u\mapsto h$ induces a homomorphism $\psi :K[u,Du]\rightarrow \mathbb{C}(\{(z - z_1)^{1/m}\})$ such that $\frac{d}{dz}\psi (u)=\psi (Du)$. If $h$ is not algebraic over $K$, then $\psi$ extends to a differential homomorphism $(F,D)\rightarrow \mathbb{C}(\{(z-z_1)^{1/e}\})$ for some multiple $e$ of $m$.

Suppose that $h$ is algebraic over $K$. Then one can extend $\psi :K[u,Du]\rightarrow \mathbb{C}(\{(z-z_1)^{1/m}\})$ to a maximal differential subring $R$ of $F$ and with values in a finite extension of the above field. It is easily seen that $R$ is some local ring $O_{X,x}$, invariant under $D$. If $t\in O_{X,x}$, then $K[s,t]\subset O_{X,x}$ and the restriction of $\psi$ to this subring produces a branched algebraic solution.

There is one case that we still have to consider, namely $t\not \in O_{X,x}$. Then $t^{-1}$ lies in the maximal ideal of $O_{X,x}$ and $\psi (t^{-1})=0$. Let $m_{X,x}$ denote the maximal ideal of $O_{X,x}$ and $K(x)$ its residue field. Then $\psi(m_{X,x}) = 0$ and $\psi$ factors as $O_{X,x}\rightarrow K(x)\rightarrow \mathbb{C}(\{(z-z_1)^{1/e}\})$ with $e$ a suitable multiple of $m$. The place of $K$ corresponding to $z_1$ is ramified in $K(x)$. As there are only finitely many of these places and finitely many $x$ with $t\not \in O_{X,x}$ the number of these places $z_1$ is finite. Therefore we found infinitely many points $z_1$ on $Z$ which are ramified for $f$.
\end{proof}

\begin{proposition}\label{prop:RegularImpliesFibredProduct}
Suppose that the derivation $D$ on $F$ is {\em regular}, i.e., every local ring $O_{X,Q}$ is invariant under $D$. Then there exists a finite extension $L$ of $K$ and a non singular curve $X_0$ over the field of constants $\mathbb{C}$ of $K$ such that $X\times _KL\cong X_0\times _{\mathbb{C}}L$.
\end{proposition}
\begin{proof} Let $g$ be the genus of $X$.

\medskip\noindent (1) {\it Suppose that $g=0$}. Then, after possibly a quadratic extension of $K$, one has $F=K(x)$. Moreover, ``$D$ regular'' is equivalent to ``$K[x]$ and $K[x^{-1}]$ are invariant under $D$.'' Hence $D$ is regular if and only if $D$ has the form $D(x)=a_0+a_1x+a_2x^2$ with all $a_i\in K$.

\medskip\noindent (2) {\it Suppose that $g=1$}. After a finite extension of $K$ we may suppose that $X$ is given by the affine equation $y^2=x(x-1)(x-a)$ with $a\in K,\ a\neq 0,1$. The assumption on $D$ implies that $D(x)=A(x)+B(x)y$ with $A(x),B(x)\in K[x]$. The completion of $O_{X,\infty}$ has the form $K[[u]]$ where $x^{-1}=u^{2}$ and $y=u^{-3}\sqrt{(1-u^2)(1-au^2)}$. The condition $D(u)\in K[[u]]$ implies $2D(u)=u^{-1}\frac{-1}{x^2}(A(1/u^2)+B(1/u^2)y)=-u^3(A(1/u^2)+B(1/u^2) \sqrt{(1-u^2)(1-au^2)}$ has no pole. This implies that $A(x)=a_0+a_1x$ and $B(x)=b_0$.

The condition $D(y)\in K[x,y]$ implies that $2yD(y)=D(x(x-1)(x-a))\in yK[x,y]$. This expression equals
\[x(x-1)(x-a)\{(\frac{1}{x}+\frac{1}{x-1}+\frac{1}{x-a})(a_0+a_1x+b_0y)-\frac{a'}{x-a}\}.\]
Hence
\[x(x-1)(x-a)\{(\frac{1}{x}+\frac{1}{x-1}+\frac{1}{x-a})(a_0+a_1x)- \frac{a'}{x-a}\}\]
lies in $yK[x,y]\cap K[x]=x(x-1)(x-a)K[x]$. Therefore $a_0=a_1=a'=0$.

This proves the statement and we find moreover that $D$ has the special form $D(x)=by$ for some $b\in K$.

\medskip\noindent (3) {\it Suppose that $X$ is a hyperelliptic curve}. Here we give an elementary proof and note that (4) contains a more abstract proof. After a finite extension of $K$ we may suppose that the affine equation of $X$ has the form $y^2=Q$ with $Q=x(x-1)(x-a_1)\cdots (x-a_m)$ where $m>1 $ is odd and $0,1,a_1,\dots ,a_m$ are distinct elements of $K$. Then $D(x)=A+By$ with $A,B\in K[x]$.

The completion of the local ring at $\infty$ can be written as $K[[u]]$ with $x^{-1}=u^2$ and  $yu^{m+2}=\sqrt{(1-u^2)(1-u^2a_1)\cdots (1-u^2a_m)}$. The condition $D(u)\in K[[u]]$ implies that $B=0$ and $A$ has degree at most one. The condition $D(y)\in K[x,y]$ implies $2yD(y)=D(Q)\in yK[x,y]\cap K[x]= QK[x]$. Now $\frac{D(Q)}{Q}\in K[x]$ implies that $A=0$ and all $a_i'=0$.

This proves the statement and shows moreover that $D$ has the special form $D(x)=D(y)=0$.

\medskip\noindent (4) {\it Suppose that $g\geq 2$}.  The derivation $D$ on $F$ extends to a $\mathbb{C}$-linear morphism of sheaves $\nabla : \Omega _{X/K}\rightarrow \Omega _{X/K}$ by the formula $\nabla (\sum f_idg_i)= \sum _i (D(f_i)dg_i+f_id(Dg_i))$. A verification that this formula is well defined is needed.

Consider an open (affine) subset $U$ of $X$. One represents $O(U)$ as $K[X_1,\dots ,X_n]/(f_1,\dots ,f_s)=K[x_1,\dots ,x_n]$. Now $D$ lifts to a derivation $D^+$ of $K[X_1,\dots ,X_n]$ having the property $D^+(f_1,\dots ,f_s)\subset (f_1,\dots ,f_s)$. Then $\Omega _{X/K}(U)$ is the free $O(U)$-module $V$ with basis $dx_1,\dots ,dx_n$ divided out by the $O(U)$-submodule $W$ spanned by the elements $\{df_i|\ i=1,\dots ,s\}$ (or equivalently by all $df$ with $f\in (f_1,\dots ,f_s)$ ). One defines $\nabla _V$ on $V$ by $\nabla _V(\sum h_idx_i)=\sum (D(h_i)dx_i+h_id(Dx_i))$ and has to verify that  $\nabla _V(W)\subset W$. This easily follows from $D^+(f_1,\dots ,f_s)\subset (f_1,\dots ,f_s)$. There results a $\nabla$ on $V/W$ with the required formula.

Now the $K$-vector space $H^0(X,\Omega _{X/K})$ is invariant under $\nabla$. This makes $H^0(X,\Omega _{X/K})$ into a differential module over $K$. The same holds for the symmetric powers $H^0(X,\Omega _{X/K}^{\otimes s})$. Thus we found line bundles $\mathcal{L}$ which are invariant under $D$, {\it i.e.}, $D$ maps $H^0(U,\mathcal{L})$ to itself for any open $U\subset X$.

Consider such an $\mathcal{L}$ which is at the same time very ample. Then the curve $X$ is equal to $Proj (\oplus _{s\geq 0}H^0(X,\mathcal{L}^{\otimes s}))$. Each $K$-vector space  $H^0(X,\mathcal{L}^{\otimes s})$ has an action $\nabla$ of $D$ on it, which makes it into a differential module. We take a Picard-Vessiot extension $U\supset K$ that trivializes these differential modules. In fact, it suffices to trivialize the differential module $H^0(X,\mathcal{L})$ since the $H^0(X,\mathcal{L}^{\otimes s})$ are images of the $s$th symmetric powers of $H^0(X,\mathcal{L})$.

Then $X\times _KU$ is the $Proj$ of $U\otimes _{\mathbb{C}}H$ where $H=\oplus _{s\geq 0}H_s$ is the homogeneous $\mathbb{C}$-algebra, generated by $H_1$, with $H_s=ker(\nabla ,H^0(X\times _KU,\mathcal{L}_{ext}^{\otimes s}))$ for all $s\geq 0$. Put $X_0=Proj(H)$, then $X\times _KU\cong X_0\times _{\bf C}U$.

It is well known that in this case there exists also a finite extension $L$ of $K$ such that $X\times _KL\cong X_0\times _{\mathbb{C}}L$. (Indeed, for this isomorphism the field $U\supset K$ can be replaced by a finitely generated $K$-algebra $R\subset L$. Further $R$ can be replaced by $R/m$ for some maximal ideal $m$).

\medskip\noindent We add here that for $g\geq 2$, the derivation $D$ has a special form. Write $F$ (after a finite extension of $K$) as $K(x,y)$ where $\mathbb{C}(x,y)$ is the function field of $X_0$. Then $D$ is zero on $\mathbb{C}(x,y)$.
\end{proof}

\begin{corollary}
Let, as before, the differential field $(F,D)$ correspond to the differential equation $f(y',y)=0$ with coefficients in the finite extension $K$ of $\mathbb{C}(z)$. If $D$ is regular, then $f$ has the PP.
\end{corollary}
\begin{proof}
\medskip\noindent (1) {\it $F/K$ with genus 0}. Then $f$ is equivalent to the Riccati equation $y'+a_0+a_1y+a_2y^2=0$ with $a_0,a_1,a_2\in K$. As in the proof of Corollary 1.3, one shows that the only movable singularities are poles.

\medskip\noindent (2) {\it $F/K$ with genus 1}. Then $f$ is equivalent to an equation of the form $(y')^2=ay(y-1)(y-\lambda )$ with $\lambda \in \mathbb{C}\setminus \{0,1\}$ and $a\in K^*$. We rewrite the latter as $(by')^2=y(y-1)(y-\lambda )$ for a suitable $b$ in a suitable $K$. Define a new derivation $d$ on $K$ by $d(f)=bf'$. The equation $d(T)=1$ has a solution in some Picard-Vessiot extension $L=K(T)$ of $K$. In general, $K(T)$ is a transcendental extension of $K$. If the extension happens to be algebraic, then we regard the differential equation as an equation over $\mathbb{C}(T)$ in the form $(\frac{d}{dT}y)^2=y(y-1)(y-\lambda )$. Then, as one knows, the only movable singularities are poles.

If $T$ is transcendental over $K$, then  the equation is again seen as an equation over the field $\mathbb{C}(T)$. The solutions are of the form ${\cal P}(T+c)$ (with $c\in \mathbb{C}$) and have only poles and no branched points.

Take a point $z_0\in \mathbb{C}$, unramified in $K$. This yields an embedding $K\subset \mathbb{C}(\{z-z_0\})$. If the image of $b$ in this field has no pole then the equation $bf'=f$ has a solution $u \in \mathbb{C}(\{z-z_0\})$. From this one finds the solutions ${\cal P}(u+c)$ in the field $\mathbb{C}(\{z-z_0\})$. This suffices to show that the only movable singularities of the equation are poles.

\medskip\noindent (3) {\it $F/K$ with genus $\geq 2$.} After a finite extension of $K$, the equation $f$ is equivalent to $y'=0$. Clearly this equation has the PP.
\end{proof}

We summarize these results as follows.

\begin{theorem}
Consider an order one differential equation $f(y',y)=0$ with coefficients in a finite extension $K$ of $\mathbb{C}(z)$. Let $(K(X),D)$ be the differential field constructed from $f$. Then $f$ has the PP if and only if, after a finite extension of $K$, the differential field $(K(X),D)$ is isomorphic to one of the following:
\begin{itemize}
\item[(a)] $(K(u),\ (a_0+a_1u+a_2u^2)\frac{d}{du})$, with $a_0,a_1,a_2\in K$ not all zero.
\item[(b)] $(K(x,y),\ f\cdot y\frac{d}{dx})$, with $y^2=x^3+ax+b,\ a,b\in \mathbb{C}$ a non singular elliptic curve and $f\in K^*$.
\item[(c)] There is a curve $X_0$ over $\mathbb{C}$ such that $X$ is isomorphic to $X_0\times _{\mathbb{C}}K$ and $D=0$ on $\mathbb{C}(X_0)$.
\end{itemize}

\smallskip
\noindent The three above cases correspond to the following differential equations:
\begin{itemize}
\item[(a')] $y'=a_0+a_1y+a_2y^2$, with $a_0,a_1,a_2\in K$ not all zero.
\item[(b')] $(y')^2=f\cdot (y^3+ay+b)$, with $a,b\in \mathbb{C}$ and $f\in K^*$.
\item[(c')] $y'=0$.
\end{itemize}
\end{theorem}

\begin{remarks}
\noindent (1) The results and proofs extend to the case where $K$ is the field of meromorphic functions on some domain of the complex sphere. Also direct limits of such fields of functions, {\it e.g}., $\mathbb{C}(\{z\})$, can be taken as base field.

\medskip\noindent (2) In Proposition \ref{prop:RegularImpliesFibredProduct}, $\mathbb{C}$ can be replaced by any algebraically closed field of characteristic 0. In fact, Matsuda has proved similar properties for more general base fields \cite{Matsuda}.

\medskip\noindent (3) For the case genus 0 and $K$ a finite extension of $\mathbb{C}(z)$ one does not need (in Proposition \ref{prop:RegularImpliesFibredProduct}) a finite extension of $K$. Indeed, $K$ is a $C_1$-field and a genus 0 curve over $K$ has a $K$-rational point.

However, one can construct a differential field $K$, a curve $X$ over $K$ of genus 0 and derivation $D$ on $K(X)$ without poles, such that a quadratic extension of $K$ is needed in Proposition \ref{prop:RegularImpliesFibredProduct}.

\medskip\noindent (4) For the genus 1 case and $K$ a finite extension of $\mathbb{C}(z)$ it can be seen that, in general, an extension of $K$ is needed.

\medskip\noindent (5) Also for genus $\geq 2$, one needs in Proposition \ref{prop:RegularImpliesFibredProduct} (in general) a finite extension of $K$. We construct some examples.

Consider a function field $\mathbb{C}(X_0)$ of a curve $X_0$ of genus $\geq 2$ on which a finite group $G\neq \{1\}$ acts faithfully and such that the curve $X_0/G$ has a strictly smaller genus. There exists a Galois extension $L\supset \mathbb{C}(z)$ with group $G$. On $R:=L\otimes \mathbb{C}(X_0)$ we let the group $G$ act by $g(a\otimes b)=g(a)\otimes g(b)$. We provide $R$ and its field of fractions $F$ with the derivation $D$ given by $D=0$ on $\mathbb{C}(X_0)$ and $D(z)=1$. Clearly $D$ commutes with the action of $G$. Moreover $\mathbb{C}(X_0)$ is the field of constants of $(F,D)$.

The ring of invariants $R^G$ and its field of fractions $F_0$ are invariant under $D$. Further $F_0$ is the function field of some curve $X$ over $\mathbb{C}(z)$ and the restriction of $D$ to $F_0$ has the properties of Proposition \ref{prop:RegularImpliesFibredProduct}. Suppose that there exists a curve $X_1$ (as in the proposition) such that $X\cong X_1\times _{\mathbb{C}}\mathbb{C}(z)$. Then the field of constants of $F_0$ is $\mathbb{C}(X_1)$. The latter is, however, equal to $\mathbb{C}(X_0)^G$ and this is the function field of the curve $X_0/G$. Now $F=L\otimes F_0$ and the two fields $F$ and $F_0$ have the same genus. Thus we found a contradiction.

\medskip\noindent (6) The equation $(y')^2=y^3+z$, which has constant $j$-invariant, does not have the PP! Indeed, consider the affine ring of the curve $\mathbb{C}(z)[s,t]$ with equation $s^2=t^3+z$ and derivation $D$ given by $D(z)=1,\ D(t)=s$. Then $D(s)=\frac{3st^2+1}{2s}$ and the latter does not belong to ${\bf C}(z)[s,t]$. As in the example of the introduction, one can make this more
explicit. A small computation shows that through every point $(y_0,z_0)$
with $z_0\neq 0$ and $y_0^3+z_0=0$ there are two branched solutions, namely
$y=y_0\pm \frac{2}{3}(z-z_0)^{3/2}+\cdots$.
\end{remarks}

\appendix
\section{Appendix}

As in Matsuda's work \cite{Matsuda}, we consider the case of differential fields $(K,\partial )$ with characteristic $p>0$. Two cases are studied here:
\begin{itemize}
\item[(1)] $\partial =0$ and
\item[(2)] $[K:K^p]=p$. We choose an element $z\in K\setminus K^p$. Then $K=K^p(z)$ and we may restrict ourselves to $\partial =\frac{d}{dz}$ ({\em i.e.}, the derivation $\partial$ is defined by $\partial (z)=1$).
\end{itemize}

\medskip The most important examples for case (2) are:
\begin{itemize}
\item[(a)] $K$ is a finite separable extension of $C(z)$, where $C$ is an algebraically closed field of characteristic $p$. The derivation $\partial$ extends  $\frac{d}{dz}$ on $C(z)$.\\
\item[(b)] $K$ is a finite separable extension of $C((z))$, where $C$ is an algebraically closed field of characteristic $p$. Further $\partial$ is the extension of $\frac{d}{dz}$ on $C((z))$.
\end{itemize}

\medskip Let $X$ be a {\it smooth}, absolutely irreducible, projective curve over $K$. Its function field is denoted by $F$. A derivation $D$ on $F$ extending $\partial$ is given which has the property that for any closed point $Q$ of $X$, the local ring $O_{X,Q}$ is invariant under $D$. {\it We will show that the classification of the possible pairs $(F,D)$ is, up to finite separable extensions of $K$, the same as in the characteristic 0 situation}.

\medskip \noindent {\it Case} (1). $D$ is a non zero element of $H^0(X,{\rm Der}_{X/K})$. It follows that there are two possibilities:

\medskip\noindent (1a) {\em Suppose that $X$ has genus 0.} After possibly replacing $K$ by a separable extension of degree 2, we may suppose that $F=K(t)$. As before one computes that $D$ must have the form $(a_0+a_1t+a_2t^2)\frac{d}{dt}$ with $a_0,a_1,a_2\in K$.

\medskip\noindent (1b) {\em Suppose that $X$ has genus 1.} The assumption that $X$ is smooth over $K$ implies that after a finite separable extension of $K$, the curve $X$ becomes an elliptic curve. Then $H^0(X,{\rm Der}_{X/K})$ has dimension 1 over $K$ and $D$ is unique up to a $K^*$-multiple. 

\medskip\noindent {\it Case} (2). \\
\medskip\noindent (2a) {\em Suppose that $X$ has genus 0.} As in case (1a), 
one obtains that, 
after possibly a separable quadratic extension of $K$, the 
function field equals $K(t)$ and $D$ has the form
 $(a_0+a_1t+a_2t^2)\frac{d}{dt}$ with $a_0,a_1,a_2\in K$.

\medskip\noindent (2b) {\em Suppose that $X$ has genus 1.} After a suitable separable extension of $K$, the function field $F$ has the form $K(s,t)$ with (compare \cite[p. 324--5]{Silverman}):\\
$s^2=t(t-1)(t-a)$ with $a\in K,\ a\neq 0,1$ for $p>2$ and for $p=2$:\\
($\alpha$) $s^2+st=t^3+a_2t^2+a_6$ with $a_6\neq 0$ or \\
($\beta$) $s^2+a_3s=t^3+a_4t+a_6$ with $a_3\neq 0$.

For $p>2$, the proof of Proposition \ref{prop:RegularImpliesFibredProduct} remains valid. Therefore $a\in K^p$ and $D$ has the form $D(t)=bs$ for some $b\in K^*$.

For $p=2$ and case ($\alpha$) one can transform the equation (using a separable extension of $K$) into $s^2+st=t^3+a$ with $a\in K^*$. Write $D(t)=A+sB$ with $A,B\in K[t]$. Then $tD(s)=s(A+Bt+t^2B)+(a+t^3)B+t^2A+a'$ and thus $A$ and $aB+a'$ are in $tK[t]$. A local parameter at $\infty$ is $u=\frac{t}{s}$. Then $D(u)=\frac{t^2A+a'+t^2Bs}{s^2}$ is supposed to have order $\geq 0$ at $\infty$. Thus both $t^2A+a'$ and $t^2Bs$ have order $\geq -6$ at $\infty$. This implies $B=0$ and thus $a'=0$ since $aB+a'\in tK[t]$. Further $A=bt$ for some $b\in K^*$. We conclude that $a\in K^2$ and $D(t)=bt$ for some $b\in K^*$ or $D=bt\frac{d}{dt}$.

For $p=2$ and case ($\beta$) one can transform the equation (using a separable extension of $K$) into $s^2+s=t^3$. The derivation $D$ has the form $D(t)=A+Bs$ with $A,B\in K[t]$ and $D(s)=t^2D(t)$. The element $u=\frac{t}{s}$ is a uniformizing parameter at $\infty$. The condition $D(u)$ has order $\geq 0$ at $\infty$ yields (after some calculation) $D(t)=a\in K$. In other words $D=a\frac{d}{dt}$ on $K(s,t)$.

\medskip\noindent (2c) {\em Suppose that the genus of $X$ is $\geq 2$.} Then the sheaf $\Omega _{X/K}^{\otimes 2}$ is very ample. The assumption on $D$ and the reasoning in the proof of Proposition \ref{prop:RegularImpliesFibredProduct} implies that $D$ induces a connection $\nabla$ on $M:=H^0(X,\Omega _{X/K}^{\otimes 2})$. As before we want to trivialize this connection, using some differential extension of the field $K$. According to \cite[Theorem 5.3]{Put}, there exists a finite separable extension $K_1\supset K$ such that the connection $K_1\otimes M$ admits a minimal Picard-Vessiot ring $A$. For notational convenience we write again $K$ for $K_1$. Put $V={\rm ker}(\nabla , A\otimes _KM)$. Thus $V$ is a vector space over the constants $K^p$ of $K$. Moreover, the natural map $A\otimes _{K^p}V\rightarrow A\otimes _KM$ is an isomorphism. As in the proof of Proposition \ref{prop:RegularImpliesFibredProduct} we find a curve $X_0$ over $K^p$ and an isomorphism $X_0\times _{K^p}A\rightarrow X\times _KA$. The ring $A$ is a local Artin ring with residue field $K$. Dividing by its maximal ideal one obtains an isomorphism $X_0\times _{K^p}K\rightarrow X$. By construction, the derivation $D$ is 0 on the function field $K^p(X_0)$. We conclude the following:\\
{\em There exists a finite separable extension $K_1\supset K$ such that $F=K(X)\subset K_1(X)=K_1(X_0)\supset K_1^p(X_0)$, where $X_0$ is a curve over $K_1^p$ and $D$ is zero on $K_1^p(X_0)$.}

\bibliography{paperorder1}

\begin{thebibliography}{BB56b}

\bibitem[BB56a]{BriotBouquet2}
C.A. Briot and J.C. Bouquet.
\newblock M\'emoire sur l'int\'egration des \'equations diff\'erentielles au
  moyen des fonctions elliptiques.
\newblock {\em Journal de l'\'Ecole (Imperiale) Polytechnique}, 36:199--254,
  1856.

\bibitem[BB56b]{BriotBouquet1}
C.A. Briot and J.C. Bouquet.
\newblock Recherches sur les propri\'et\'es des fonctions d\'efinies par des
  \'equations diff\'erentielles.
\newblock {\em Journal de l'\'Ecole (Imperiale) Polytechnique}, 36:133--198,
  1856.

\bibitem[Bie53]{Bieberbach}
L.~Bieberbach.
\newblock {\em Theorie der Gew\"ohnlichen Differentialgleichungen: auf
  Funktionentheoretische Grundlage Dargestellt}.
\newblock Springer, Berlin, 1953.

\bibitem[Fuc84]{Fuchs}
L.~Fuchs.
\newblock \"uber differentialgleichungen, deren integrale feste
  verzweigungspunkte besitzen.
\newblock {\em K\"oniglichen Preussischen Akademie der Wissenschaften},
  32:699--719, 1884.

\bibitem[Hil76]{Hille}
E.~Hille.
\newblock {\em Ordinary Differential Equations in the Complex Domain}.
\newblock Wiley-Interscience, New York, 1976.

\bibitem[Inc56]{Ince}
E.L. Ince.
\newblock {\em Ordinary Differential Equations}.
\newblock Dover Publications, Inc, New York, 1956.

\bibitem[Mal41]{Malmquist}
J.~Malmquist.
\newblock Sur les fonctions \'a un nombre fini de branches satisfaisant \'a une
  \'equation diff\'erentielle du premier ordre.
\newblock {\em Acta Math.}, 74:175--196, 1941.

\bibitem[Mat80]{Matsuda}
M.~Matsuda.
\newblock {\em First Order Algebraic Differential Equations: A Differential
  Algebraic Approach}.
\newblock Lecture Notes in Mathematics. Springer-Verlag, Berlin, 1980.

\bibitem[Pai88]{Painleve1888}
P.~Painlev\'e.
\newblock Sur les equations diff\'erentielles du premier ordre.
\newblock {\em Comptes-Rendus Acad. Sc. Paris}, 107:221--224, 320--323,
  724--726, 1888.

\bibitem[Pai95]{Painleve}
P.~Painlev\'e.
\newblock {\em Le\c{c}ons sur la Th\'eorie Analytique des \'Equations
  Diff\'erentielles, Profess\'ees a Stockholm}.
\newblock Hermann, Paris, 1895.
\newblock Reprinted, Oeuvres de Paul Painlev\'e, vol. I (Editions du CNRS,
  Paris, 1973).

\bibitem[Poi85]{Poincare}
H.~Poincar\'e.
\newblock Sur un th\'eor\`eme de m. fuchs.
\newblock {\em Acta Mathematica}, 7:1--32, 1885.

\bibitem[Sil86]{Silverman}
J.H. Silverman.
\newblock {\em The Arithmetic of Elliptic Curves}.
\newblock Springer-Verlag, New York, 1986.

\bibitem[vdP95]{Put}
M.~van~der Put.
\newblock Differential equations in characteristic $p$.
\newblock {\em Compositio Mathematica}, 97:227--251, 1995.

\end{thebibliography}
\bibliographystyle{alpha}
%\nocite{*}
\end{document}